\documentclass[journal,12pt,twoside,onecolumn]{IEEEtran}
\usepackage[dvips,final]{graphicx}
\usepackage{amsmath,epsfig}
\usepackage{color}
\usepackage{algorithm}
\usepackage{algorithmic}
\usepackage{textcmds}
\usepackage{amsthm}
\theoremstyle{definition}

\newtheorem{theorem}{Theorem}[section]

\newtheorem{proposition}[theorem]{Proposition}

\begin{document}

\title{N-sphere chord length distribution}

\author{Panagiotis~Sidiropoulos
\thanks{P.~Sidiropoulos is with the Imaging Group, Mullard Space Science Laboratory / University College London, UK, ~p.sidiropoulos@ucl.ac.uk.}}
\maketitle

\begin{abstract}
\noindent This work studies the chord length distribution, in the case where both ends lie on a $N$-dimensional hypersphere ($N \geq 2$). Actually, after connecting this distribution to the recently estimated surface of a hyperspherical cap \cite{SLi11}, closed-form expressions of both the probability density function and the cumulative distribution 
function are straightforwardly extracted, which are followed by a discussion on its basic properties, among which its dependence from the hypersphere dimension. Additionally, the distribution of the dot product of unitary vectors is estimated, a problem that is related to the chord length.
\end{abstract}


\IEEEdisplaynotcompsoctitleabstractindextext

\IEEEpeerreviewmaketitle

\section{Introduction}

The chord length distribution, when its ends lie on a locus, is a rather challenging problem, which still is only partially solved at present. However, recently a number of researchers have established closed-form solutions for a number of special locus types, such as the chord length distribution of a regular polygon \cite{UBasel14}, a parallelogram \cite{SRen12} and a cube \cite{JPhilip07}. Additionally, properties of the distribution in more generic scenarios have been examined and modelled, such as the average chord length in a compact convex subset of a n-dimensional Euclidean space \cite{BBurgstaller09}. 

The chord length distribution of a hypersphere is still not available, even though it presents not only mathematical interest, but also can be employed in a multitude of applications. For example, since by default normalised vectors lie on a hypersphere of radius $R=1$, the $N$-sphere chord length distribution could be used as a ``randomness'' benchmark to give a hint whether a set of normalised vectors is selected uniformly and independently. On the other hand, in applications where the input data are known only through their distances, the detection of $N$-dimensional hyperspheres through their ``chord distribution footprint'' could provide a low boundary in the dimension of the unknown input data. Moreover, this distribution could be used to train classifiers only with positive examples in cases where the negative examples can be modelled as uniform and independently selected distances.

Recently, S. Li achieved to produce a closed-form expression of the surface of a hyperspherical cap \cite{SLi11}. In this formula the hyperspherical-cap surface is given as a fraction of the total hypersphere surface. In this work it is proven that the chord length distribution in a hypersphere reduces to the ratio of the hyperspherical-cap surface over the total hypersphere surface. This proof, along with the resulting probability density function (pdf) and cumulative distribution 
function (cdf) are introduced in Section \ref{tdonvd}, followed by the properties of the new distribution, as well as its dependence from the hypersphere dimension $N$, in Section \ref{subsec:bp}. The (determined from the chord length distribution) unitary vector dot product distribution is examined in Section \ref{sec:uvdp}, while, Section \ref{sec:cr} concludes this work.

\section{N-sphere chord length distribution}
\label{tdonvd}

Let $p_i = \{p_{i1}~p_{i2}~p_{i3}~...~p_{iN}\}$ be points selected uniformly and independently from the surface of a $N$-dimensional hypersphere of radius $R$, i.e., $\forall ~ p_i, ~p_{i1}^2 + p_{i2}^2 +...p_{iN}^2=R^2$. The pairwise Euclidean distances $d(i,j), i,j \in \{1,2,...N\}, i \neq j$ of $p_i$, $p_j$ generate a set $d_k$ of distances ($k~=~N(N-1)/2$). We would like to estimate the distribution that the $d_k$ distance values follow. Since the distances $d_k$ are invariant to the coordinate system, it is assumed that this is selected so as the first end of the chord is always
$\{0,0,0...0,R\}$. The second end of the chord determines the chord length. 

When $N=2$ the hypersphere is a circle and the objective degenerates to the known (e.g., \cite{EWeisstein2}) circle chord length distribution. As a matter of fact, the pdf $f_2(d)$ and the cdf $F_2(d)$ are given by the following formulas:

\begin{equation}
\label{eq:circle_pdf} f_2(d) = \frac{1}{\pi}\frac{1}{\sqrt{1-\frac{d^2}{2R^2}}}
\end{equation}

\begin{equation}
\label{eq:circle_cdf} F_2(d) = \frac{cos^{-1}(1-\frac{d^2}{2R^2})}{\pi}
\end{equation}
However, in the general case (i.e., when $N \geq 2$) there are no closed-forms expressions giving neither the pdf $f_N(d)$ nor the cdf $F_N(d)$. In this paper we introduce these formulas. In order to achieve this, we need to prove that

\begin{proposition}
\label{Lemma1}
The locus of the $N$-sphere points that have a constant distance $d$ from a fixed point $p$ on it is a $(N-1)$-sphere of radius $a=\sqrt{d^2-\frac{d^4}{4R^2}}$. 
\end{proposition}

\begin{proof}
The fixed point $p$ have coordinates $\{0,0,0...0,R\}$. Then the points of the hypersphere that have distances $d$ from $p$ have coordinates $(x_{1}, x_{2}, x_{3},...,x_{N})$ that satisfy the following two equations:

\begin{equation}
\label{eq:111} x_{1}^2 + x_{2}^2 +...+x_{N}^2=R^2
\end{equation}

\begin{equation}
x_{1}^2 + x_{2}^2 +...+(x_{N}-R)^2=d^2
\end{equation}

By the above equations it follows that

\begin{equation}
R^2-2Rx_{N}=d^2-R^2,
\end{equation}

By replacing $x_{N}$ to Eq. (\ref{eq:111}) we get that

\begin{equation} \label{eq:1121}
x_{1}^2 + x_{2}^2 +...+x_{(N-1)}^2=d^2 - \frac{d^4}{4R^2}
\end{equation}
\end{proof}

The $(N-1)$-sphere of Eq. (\ref{eq:1121}) is the intersection of the $N$-sphere with the hyperplane $L: x_{N}=R-\frac{d^2}{2R}$. Since $\frac{\partial x_{N}}{\partial d} \leq 0$, all points of the N-sphere with distance $D$ from $p$, $D < d$, lie ``northern'' to the $(N-1)$-sphere (i.e., have larger $x_{N}$ value than the points in $L$), while all points with distance $D$ from $p$, $D > d$, lie ``southern'' to the $(N-1)$-sphere. It can be deduced that the above hyperplane cuts the hypersphere into two parts, each defined by the distance from $p$ being smaller or larger than $d$. By default, a hyperspherical cap is defined as the portion of a hypersphere that is cut by a hyperplane, thus the two above parts are hyperspherical caps. Consequently, it was shown that

\begin{proposition}
\label{theorem1}
The locus of the $N$-sphere points that have distance $D$, $D \leq d$ from a point on it is a hyperspherical cap of radius 
$a=\sqrt{d^2-\frac{d^4}{4R^2}}$. 
\end{proposition}

Since the maximum distance $d$ the cap radius $a$ and the cap height $h$ form a right triangle (Fig. \ref{fig:chord_circle}), the height of the cap is

\begin{equation}
\label{eq:height_cap} h = \frac{d^2}{2R}
\end{equation}

\begin{figure*}[htb]
\centering
\includegraphics[width=0.5\textwidth]{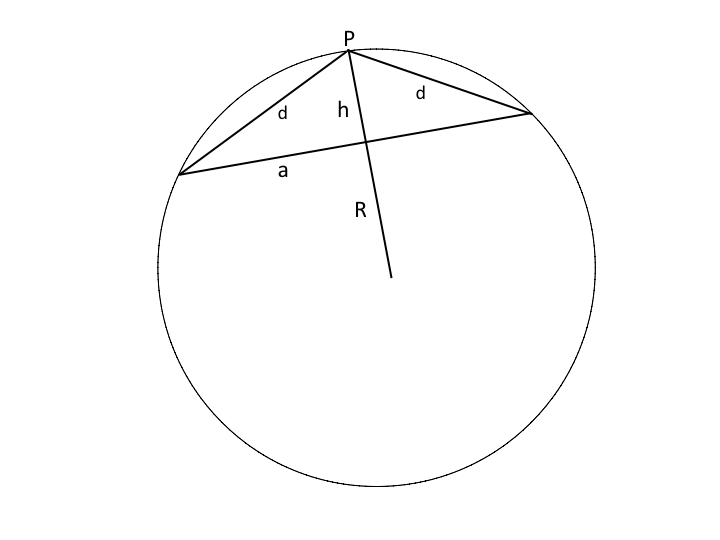}
\caption {A hyperspherical cap and the relation of the maximum distance $d$ from a point $P$, the hyperspherical cap height $h$ and its radius $a$.}
\label{fig:chord_circle}
\end{figure*}

Recently it was proven from S. Li \cite{SLi11} that the surface $A^{cap}_N(R)$ of a hyperspherical cap (that is smaller than a hyper-hemisphere, i.e., $h \leq R$) is given by the following formula:

\begin{equation}
\label{eq:li_surface} A^{cap}_N(R) = \frac{1}{2}A_N(R)I_{sin^2\phi}(\frac{N-1}{2},\frac{1}{2})
\end{equation}

In Eq. (\ref{eq:li_surface}), $N$ is the hypersphere dimension, $R$ its radius, $A_N(R)$ the hypersphere surface, $\phi$ the colatitude angle \cite{SLi11} and $I$ the regularised incomplete beta function given by

\begin{equation}
\label{eq:incomplete_beta} I_x(a,b) = \frac{B(x;a,b)}{B(a,b)} = \frac{\int_0^x t^{a-1} (1-t)^{b-1}dt}{\int_0^1 t^{a-1} (1-t)^{b-1}dt}
\end{equation}

The colatitude angle, the sphere radius and the hyperspherical cap satisfy the following \cite{SLi11}:

\begin{equation}
\label{eq:colatitude} h = (1-cos\phi)R
\end{equation}

By replacing $h$ from Eq. (\ref{eq:height_cap}) it follows that

\begin{equation}
\label{eq:colatitude2} \phi = cos^{-1} (1-\frac{d^2}{2R^2})
\end{equation}

Finally, from Proposition \ref{theorem1}, Eq. (\ref{eq:li_surface}) and Eq. (\ref{eq:colatitude2}) it is deduced that

\begin{proposition}
The cumulative distribution function of the $N$-sphere chord length is

\begin{equation} \label{eq:hypersphere_cdf} 
\centering
\begin{split}
P(D\leq d) = F_N(d) = \frac{1}{2}I_{\frac{d^2}{R^2}-\frac{d^4}{4R^4}}(\frac{N-1}{2},\frac{1}{2}), d<\sqrt{2}R \\
P(D\leq d) = F_N(d) = 1 - \frac{1}{2}I_{\frac{d^2}{R^2}-\frac{d^4}{4R^4}}(\frac{N-1}{2},\frac{1}{2}), d \geq \sqrt{2}R
\end{split}
\end{equation}
\end{proposition}

\begin{proposition}
The probability density function of the $N$-sphere chord length is:

\begin{equation} \label{eq:hypersphere_pdf}
f_N(d) = \frac{d}{R^2B(\frac{N-1}{2},\frac{1}{2})}(\frac{d^2}{R^2}-\frac{d^4}{4R^4})^{\frac{N-3}{2}}
\end{equation}
\end{proposition}

\section{Some properties of the normalised vector distance distribution}
\label{sec:spotnvdd}

Equations (\ref{eq:hypersphere_cdf}) and (\ref{eq:hypersphere_pdf}) give the cdf and the pdf of the $N$-sphere chord length distribution, respectively. In this section, some of its basic properties are examined. Furthermore, we discuss how these properties vary with $N$, especially when $N \rightarrow \infty$.

\subsection{Basic properties}
\label{subsec:bp}

Eq. (\ref{eq:hypersphere_cdf}) suggests that the pdf $f_3(d)$, $f_4(d)$ and $f_5(d)$ are as follows:

\begin{equation} \label{eq:3d_pdf}
f_3(d) = \frac{d}{2R^2},
\end{equation}

\begin{equation} \label{eq:4d_pdf}
f_4(d) = \frac{4d^2}{\pi R^3} \sqrt{1-\frac{d^2}{4R^2}},
\end{equation}

\begin{equation} \label{eq:5d_pdf}
f_5(d) = \frac{3d^3}{4R^4}(1-\frac{d^2}{4R^2})
\end{equation}

while the cdf $F_3(d)$, $F_4(d)$ and $F_5(d)$ are

\begin{equation} \label{eq:3d_cdf}
F_3(d) = \frac{d^2}{4R^2},
\end{equation}

\begin{equation} \label{eq:4d_cdf} 
F_4(d) = \frac{cos^{-1}(1-\frac{d^2}{2R^2})}{\pi}-\frac{2}{\pi}(1-\frac{d^2}{2R^2})\sqrt{1-\frac{d^2}{4R^2}},
\end{equation}

\begin{equation} \label{eq:5d_cdf}
F_5(d) = \frac{3d^4}{16R^4}-\frac{3d^6}{96R^6},
\end{equation} 

When $d=\sqrt{2}R$, then $d^2/R^2-d^4/4R^4=1$. Since $I_1(a,b)=1$, Eq. (\ref{eq:hypersphere_cdf}) implies that $F_N(\sqrt{2}R) = 0.5$, i.e., that

\begin{proposition}
\label{proposition1}
The median value of the $N$-sphere chord length distribution is independent of $N$ and equal to $\sqrt{2}R$.
\end{proposition}
 
The probability density functions and cumulative distribution functions for $N = 2, 3, 4, 8, 16, 32$ are shown in Figs. \ref{fig:cdf_examples} and \ref{fig:pdf_examples}, respectively. In order to estimate the distribution moments, we use the variable transform $d^2/4R^2 =u$, which leads to the following formula:

\begin{equation} \label{eq:moments_equation}
\mu_k = \frac{2^{k+N-2}}{B(\frac{N-1}{2},\frac{1}{2})} B(\frac{k+N-1}{2},\frac{N-1}{2}) R^k
\end{equation}

\begin{figure*}[htb]
\begin{center}
\begin{tabular}{ccc}
\includegraphics[width=0.33\textwidth]{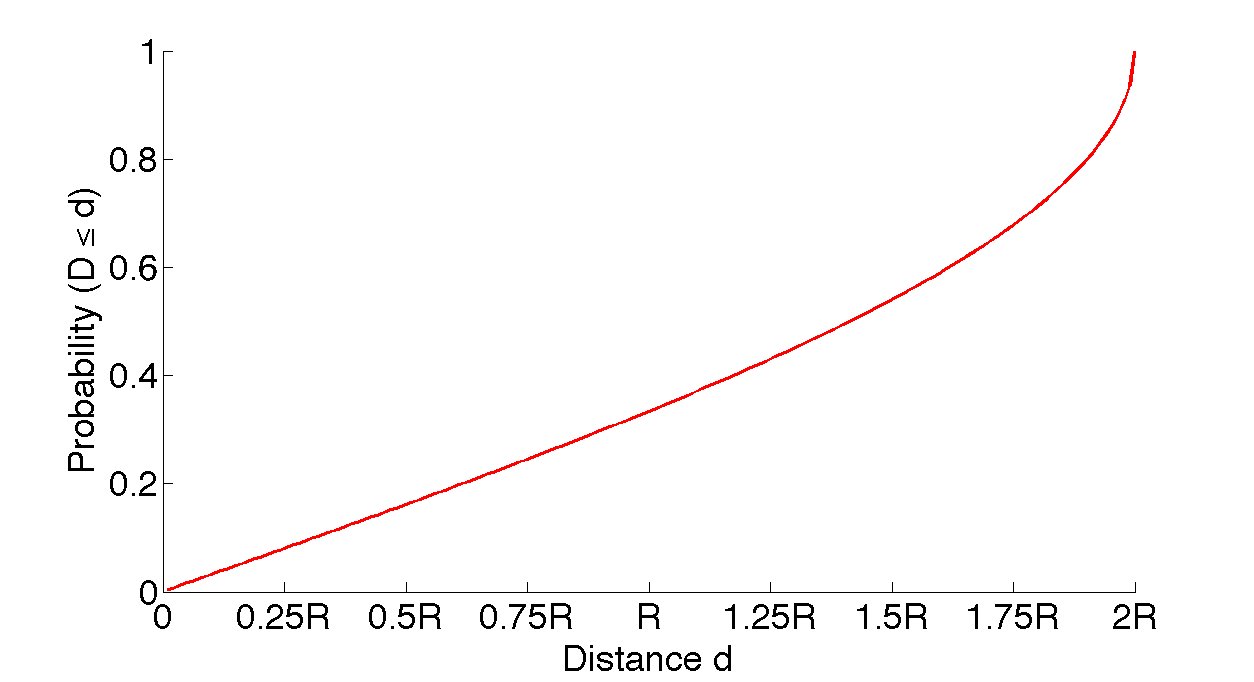} & \includegraphics[width=0.33\textwidth]{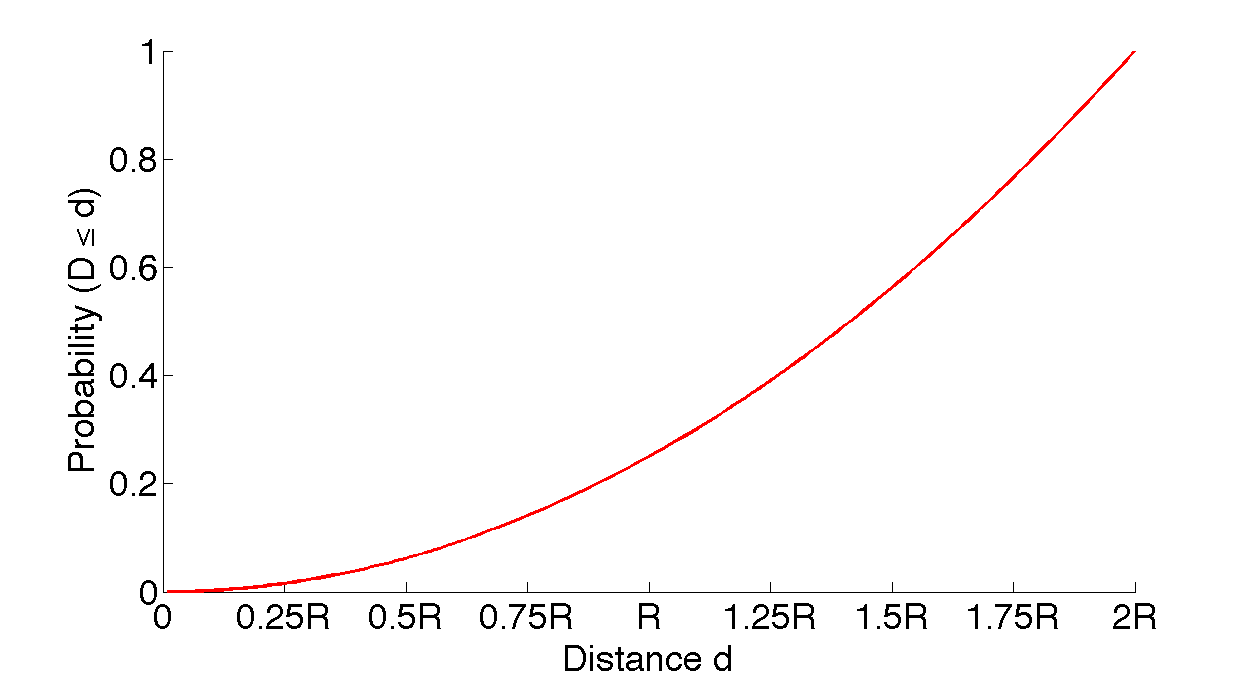} & \includegraphics[width=0.33\textwidth]{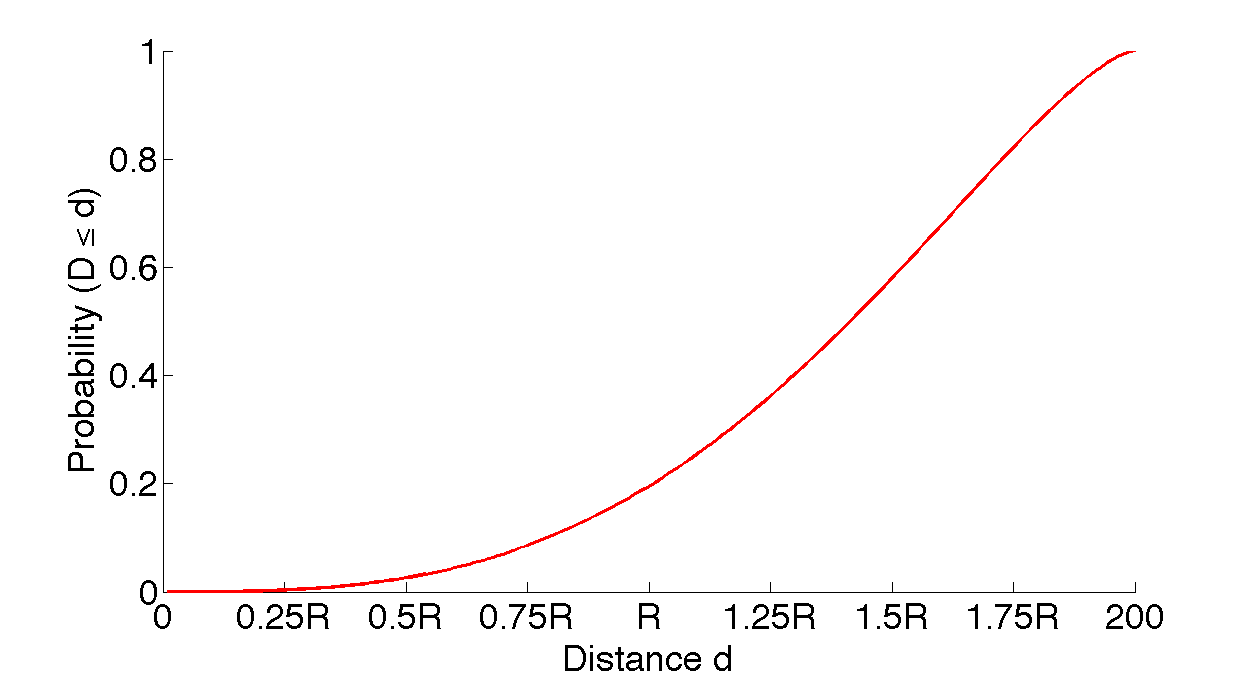} \\
(a) & (b) & (c) \\
\includegraphics[width=0.33\textwidth]{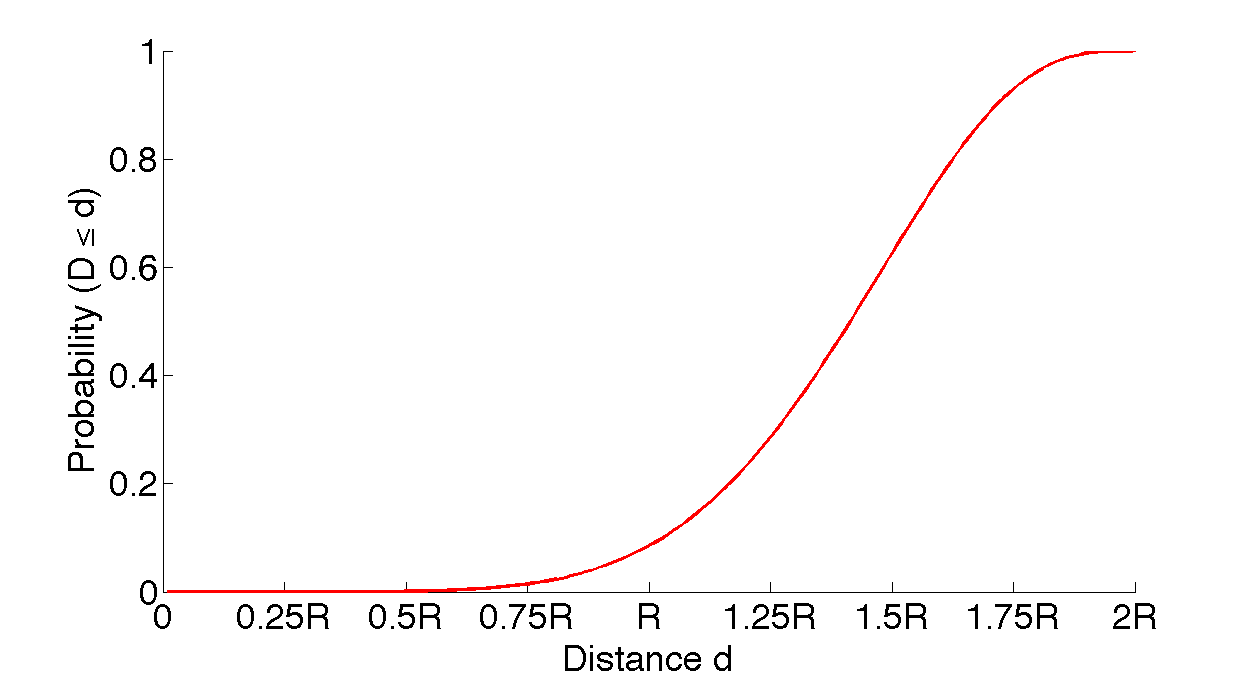} & \includegraphics[width=0.33\textwidth]{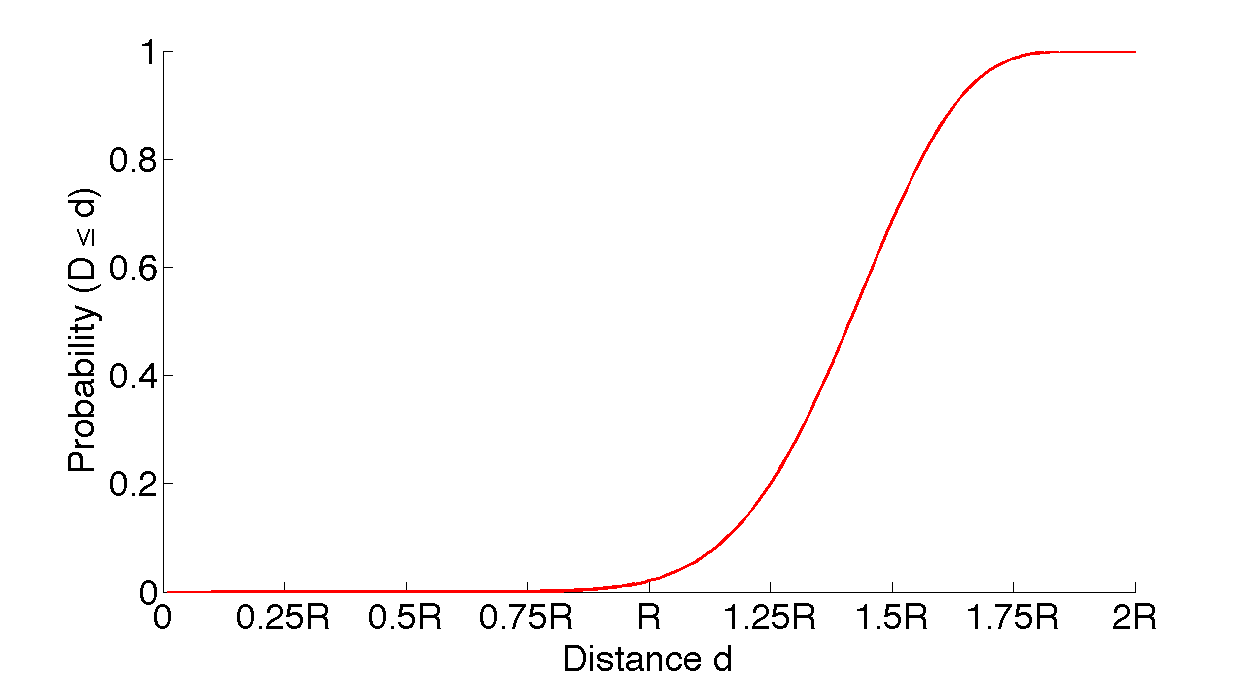} & \includegraphics[width=0.33\textwidth]{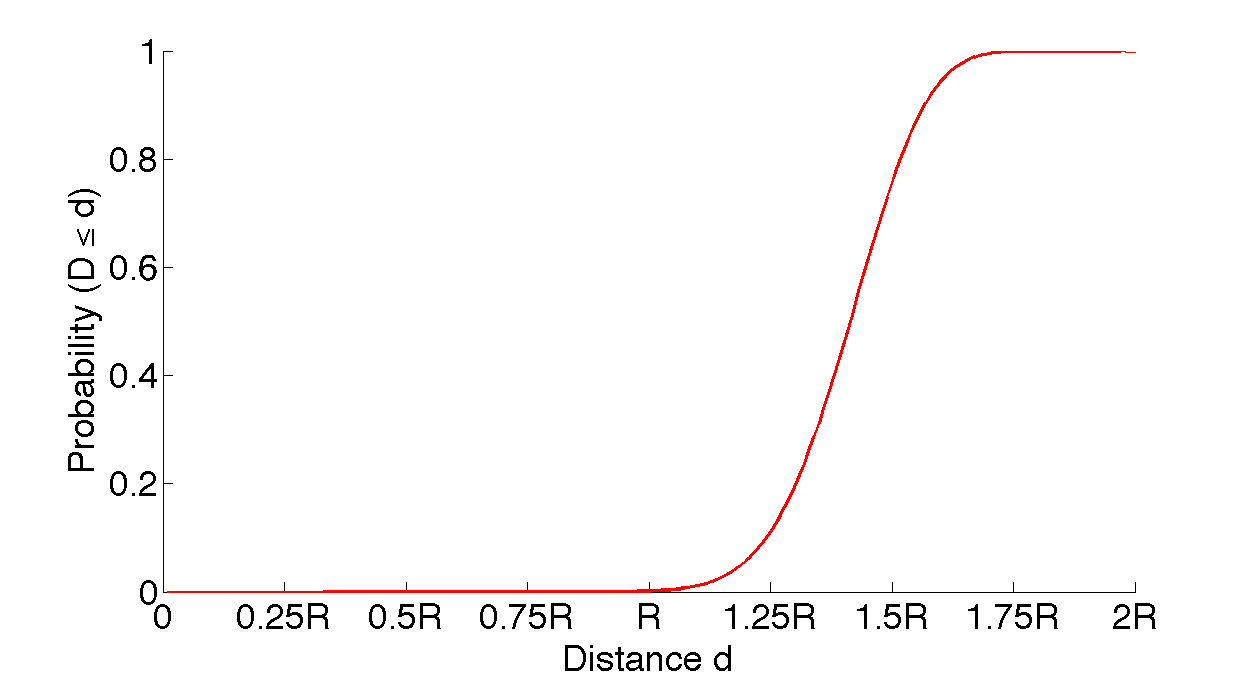} \\
(d) & (e) & (f) \\
\end{tabular}
\end{center}
\caption {The cumulative distribution functions. (a) $F_2(d)$ (b) $F_3(d)$ (c) $F_4(d)$ (d) $F_8(d)$ (e) $F_16(d)$ (f) $F_32(d)$.}
\label{fig:cdf_examples}
\end{figure*}

\begin{figure*}[htb]
\begin{center}
\begin{tabular}{ccc}
\includegraphics[width=0.33\textwidth]{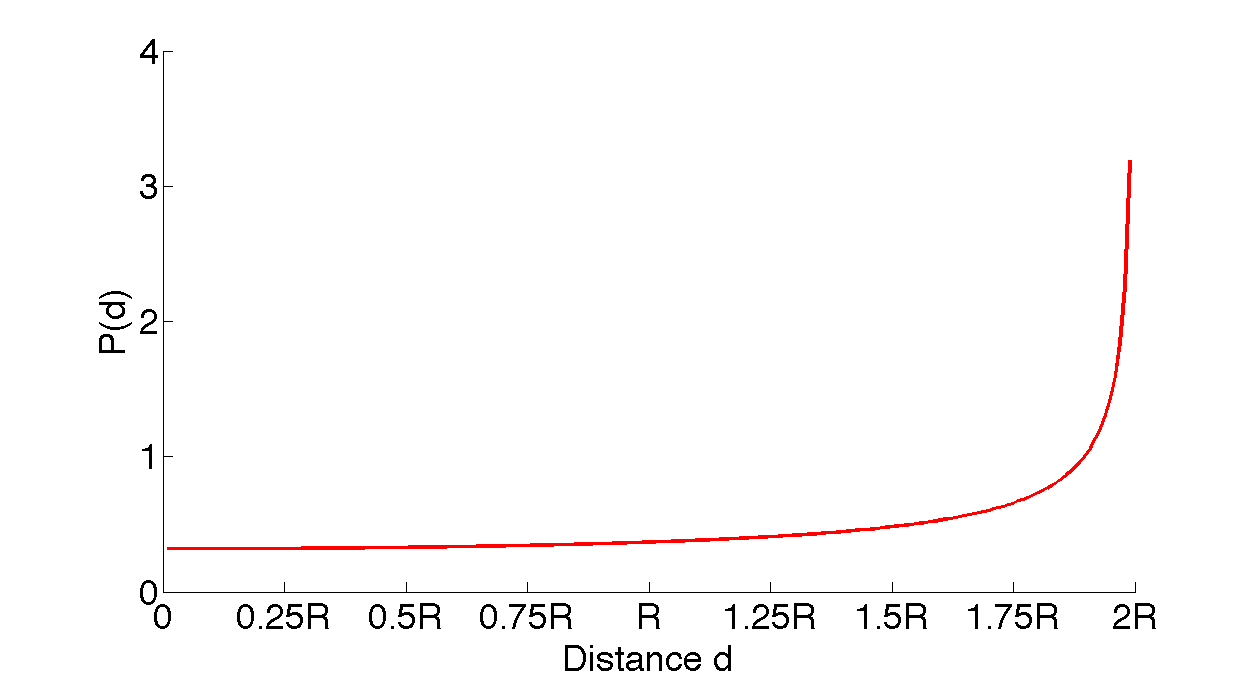} & \includegraphics[width=0.33\textwidth]{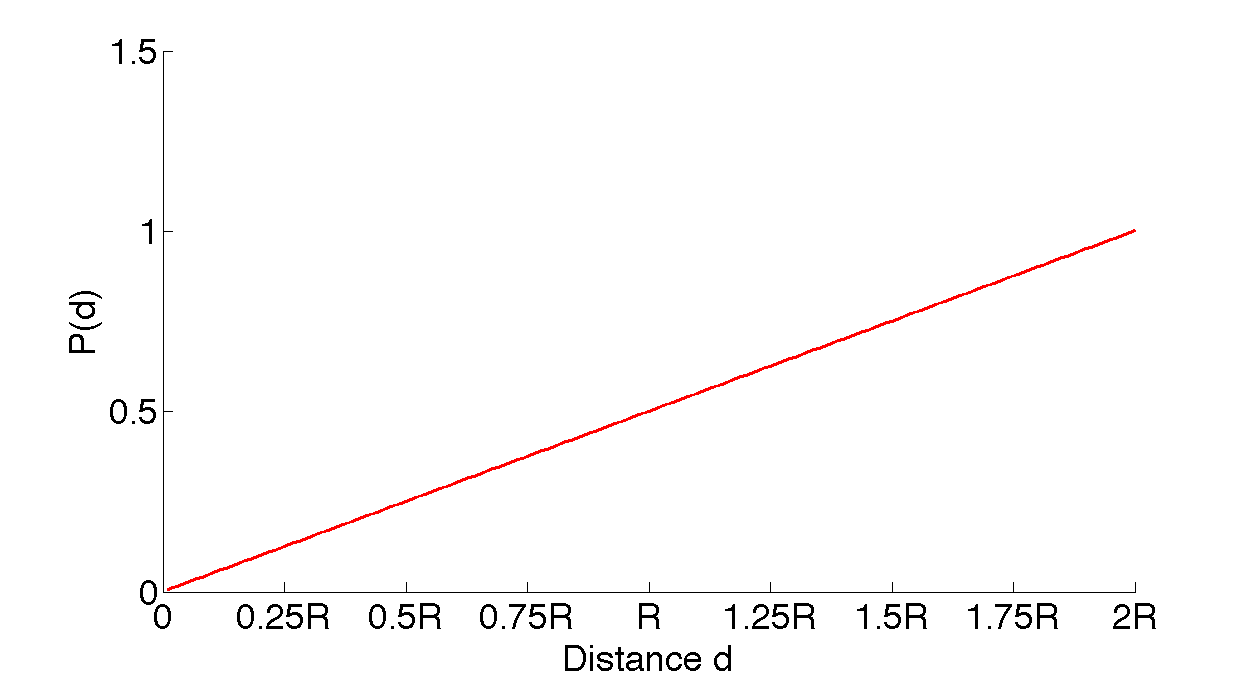} & \includegraphics[width=0.33\textwidth]{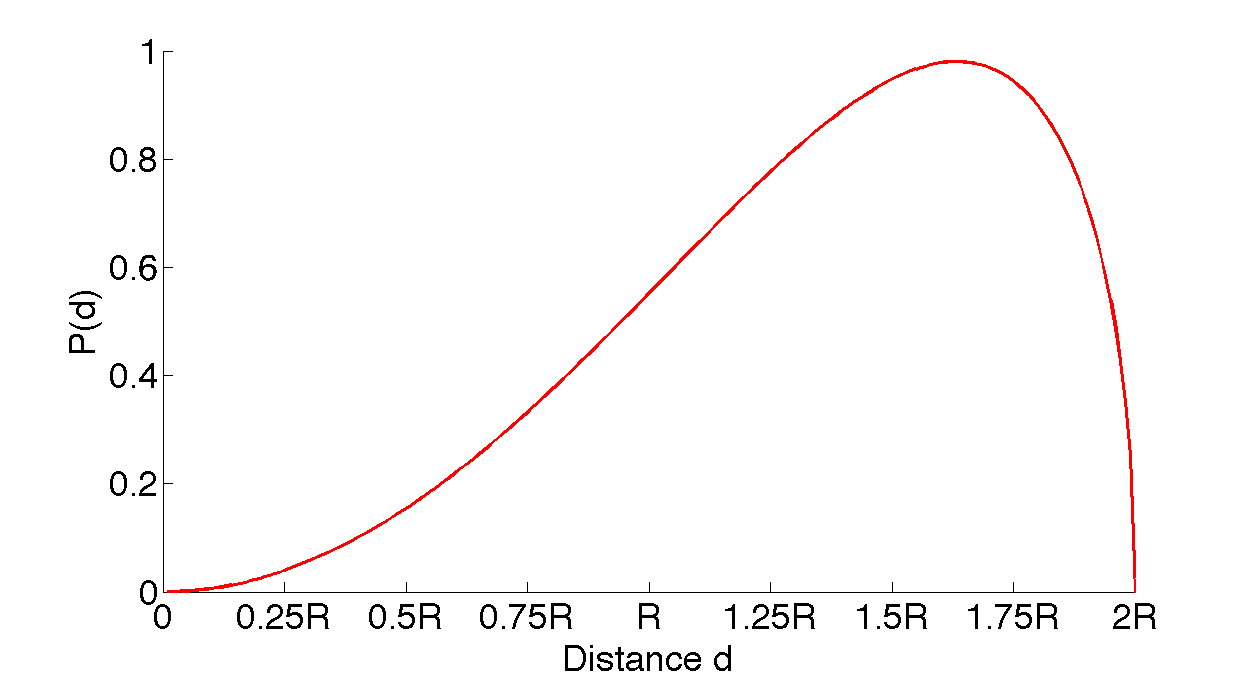} \\
(a) & (b) & (c) \\
\includegraphics[width=0.33\textwidth]{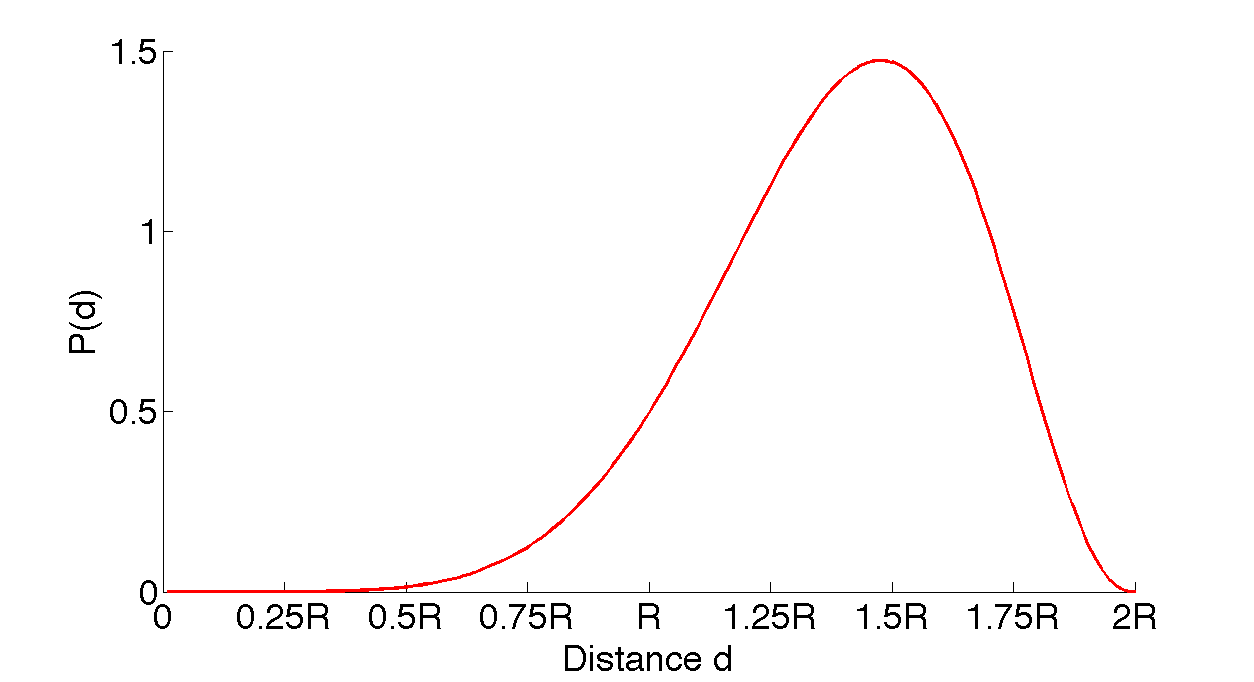} & \includegraphics[width=0.33\textwidth]{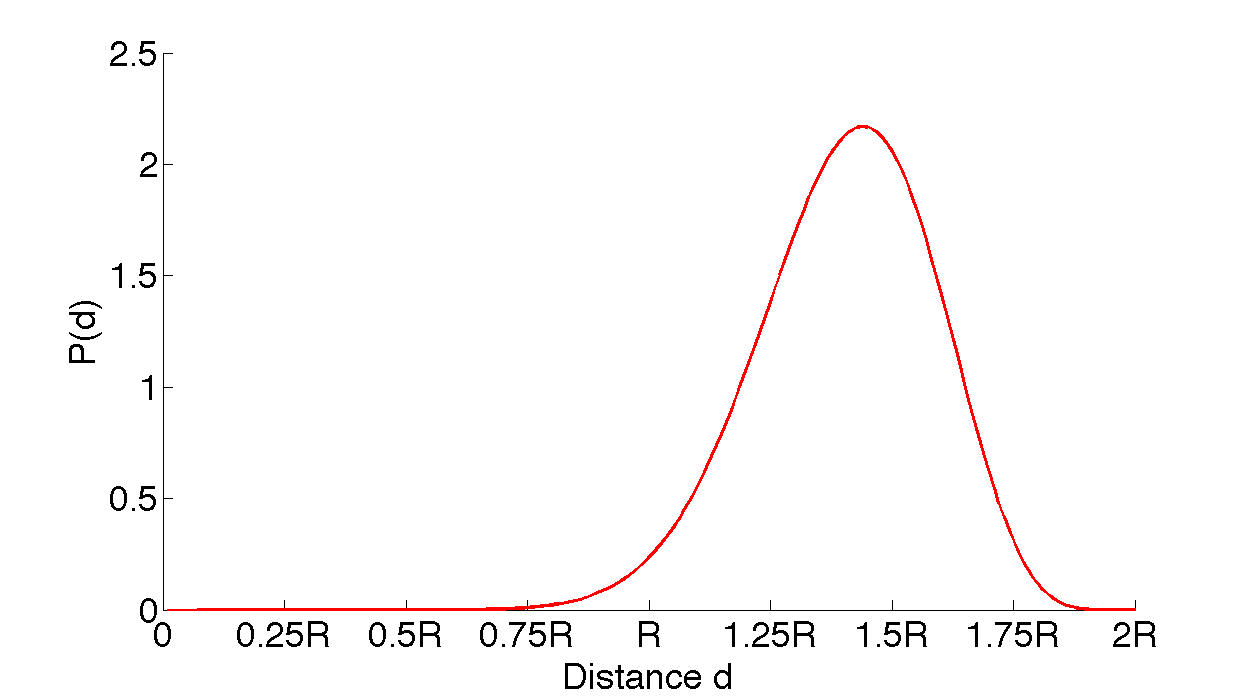} & \includegraphics[width=0.33\textwidth]{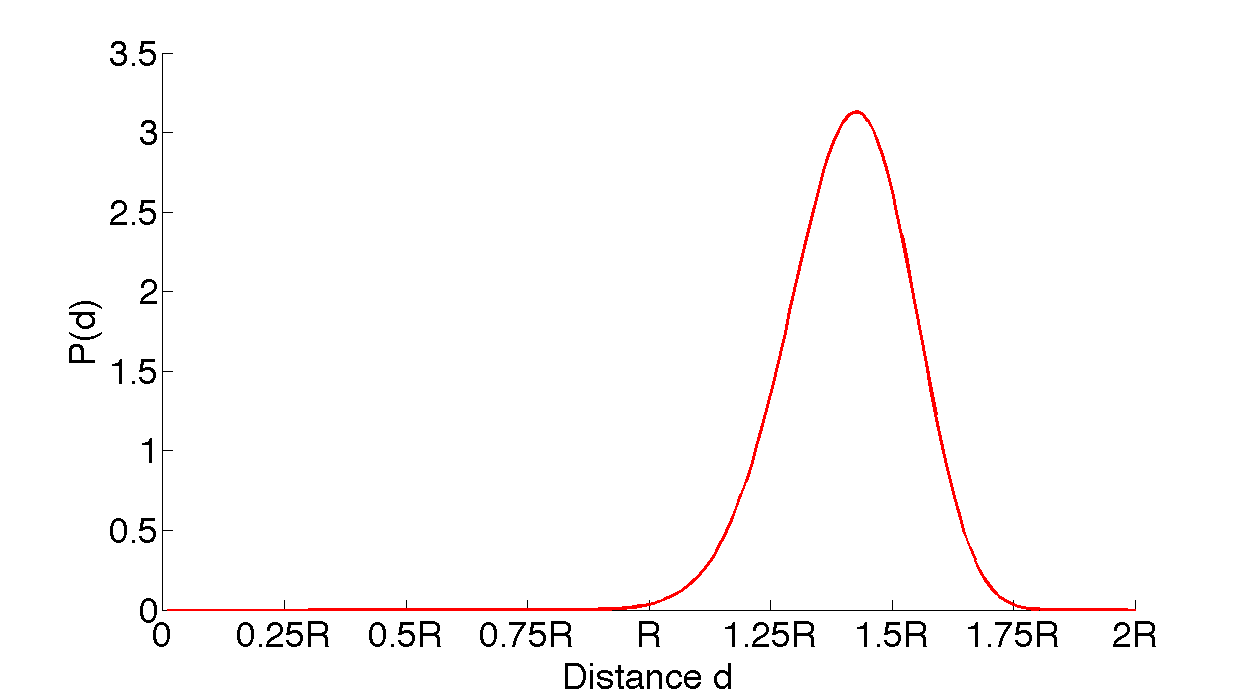} \\
(d) & (e) & (f) \\
\end{tabular}
\end{center}
\caption {The probability density functions. (a) $f_2(d)$ (b) $f_3(d)$ (c) $f_4(d)$ (d) $f_8(d)$ (e) $f_16(d)$ (f) $f_32(d)$.}
\label{fig:pdf_examples}
\end{figure*}

Eq. (\ref{eq:moments_equation}) gives the $k-th$ order moment of the $N$-sphere chord length distribution. For $k=1$, the mean value is given by

\begin{proposition}
\label{proposition2} 
The mean value $\mu$ of the $N$-sphere chord length distribution is
\begin{equation} \label{eq:moments_equation_3}
\mu= \frac{B(\frac{N}{2},\frac{N}{2})}{B(N-\frac{1}{2},\frac{1}{2})}2^{N-1}R
\end{equation}
\end{proposition}

On the contrary, for the second order moment $\mu_2$ the following proposition stands:

\begin{proposition}
\label{theorem2}
The second order moment $\mu_2$ of the $N$-sphere chord length distribution is independent of $N$ and equal to $2R^2$.
\end{proposition}

\begin{proof}
By replacing $k=2$ in Eq. (\ref{eq:moments_equation}) we obtain the following:

\begin{equation} \label{eq:moments_equation_2}
\mu_2 = 2^NR^2 \frac{B(\frac{N+1}{2},\frac{N-1}{2})}{B(\frac{N-1}{2},\frac{1}{2})} = 2^NR^2 \frac{\Gamma(\frac{N}{2})\Gamma(\frac{N+1}{2})}{\Gamma(\frac{1}{2})\Gamma(N)}
\end{equation}
where $\Gamma$ is the gamma function. Since it is known that for all $z$ the following property stands \cite{EWeisstein3}:

\begin{equation} \label{eq:gamma_ratio}
\frac{\Gamma(z)\Gamma(z+\frac{1}{2})}{\Gamma(\frac{1}{2})\Gamma(2z)} = 2^{1-2z}
\end{equation}

By replacing $z=\frac{N}{2}$ in Eq. (\ref{eq:moments_equation_2}) it follows that $\mu_2=2R^2$.
\end{proof}

Propositions \ref{theorem2} and \ref{proposition2} imply that the variance $\sigma^2$ is

\begin{equation} \label{eq:gamma_ratio}
\sigma^2 = (2-\frac{B^2(\frac{N}{2},\frac{N}{2})}{B^2(N-\frac{1}{2},\frac{1}{2})}2^{2N-2})R^2 = (2-\frac{\Gamma^4(\frac{N}{2})}{\pi \Gamma^2(N-\frac{1}{2})}2^{2N-2})R^2
\end{equation}

Finally, it should be noted that the Bertrand problem \cite{APapoulis84}, which refers to the probability $P_R$ of a random chord being larger than the radius, can be answered by replacing $d=R$ in Eq. (\ref{eq:hypersphere_cdf}), i.e., 

\begin{equation} \label{eq:bertrand}
P_R = P(d\leq R) = \frac{1}{2}I_{\frac{3}{4}}(\frac{N-1}{2},\frac{1}{2})
\end{equation}
Fig. \ref{fig:bertrand} demonstrates $P_R$ as a function of $N$. More results about the chord length distribution dependence from the hypersphere dimension $N$ are given in the next sub-section.

\begin{figure*}[htb]
\centering
\includegraphics[width=0.8\textwidth]{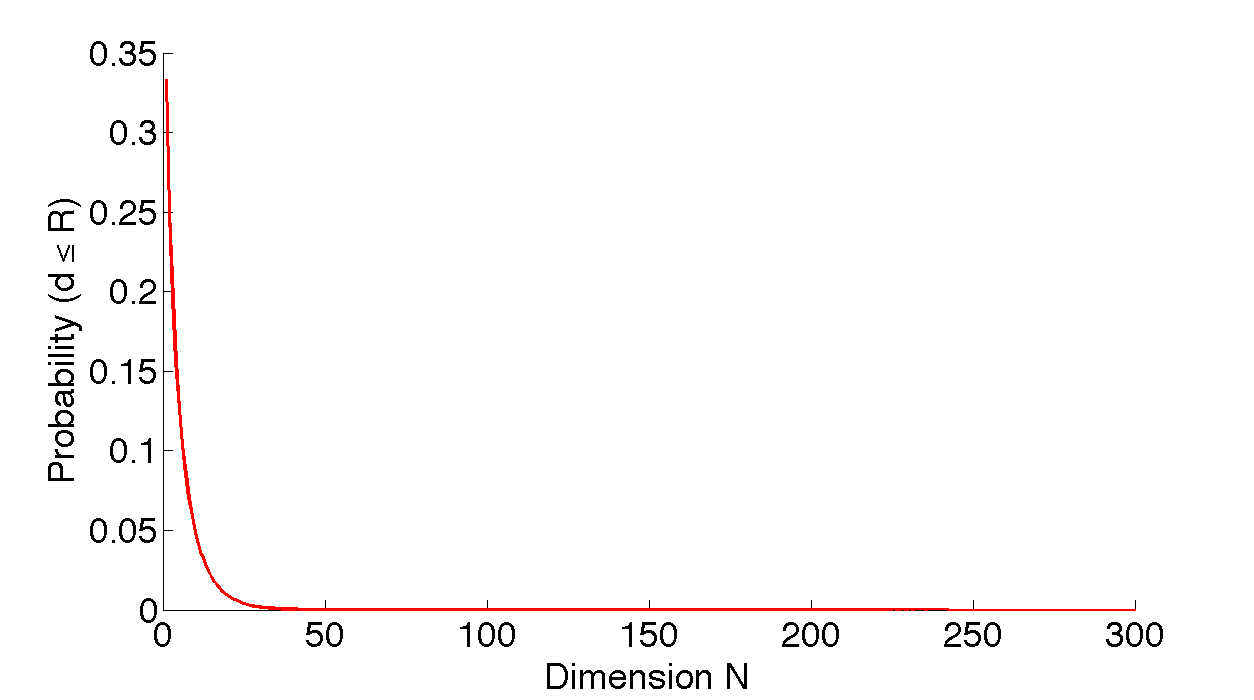}
\caption {The probability that the chord length of a hypersphere is smaller than the hypersphere radius, as a function of the hypersphere dimension.}
\label{fig:bertrand}
\end{figure*}

\subsection{Dependence from the hypersphere dimension}
\label{subsec:dfthd}

The cdf of Eq. (\ref{eq:hypersphere_cdf}) can be recursively estimated using the following formula for regularised incomplete beta functions:

\begin{equation} \label{eq:incomplete_recursice}
I_x(a+1,b) = I_x(a,b) - \frac{x^a(1-x)^b}{aB(a,b)}
\end{equation}
In the $N$-sphere chord length distribution case $a=(N-1)/2$, $b=1/2$ and $x=d^2/R^2-d^4/4R^4$. $x$ is a variable that takes values in the interval [0,1] when $d\leq \sqrt{2}R$. This signifies that the second term of the right part of Eq. (\ref{eq:incomplete_recursice}) is always positive in the interval $d \leq \sqrt{2}R$, i.e., that the cdf scores that correspond to a fixed $d \leq \sqrt{2}R$ reduce with $N$. Similarly, when $d \geq \sqrt{2}R$ the second term of the right part of Eq. (\ref{eq:incomplete_recursice}) is always negative, which means that the cdf scores that correspond to a fixed $d\leq \sqrt{2}R$ increase with $N$. Summarising:

\begin{proposition}
\label{proposition3}
If $P_i(D \leq d)$ is the probability that a chord length is smaller than $d$, when both chord ends lie on a hypersphere of radius $R$ and dimension $N_i$ then \\
\begin{itemize}
\item If $d \leq \sqrt{2}R$, then $P_1(D \leq dR) < P_2(D \leq dR)$ $\Leftrightarrow N_1 < N_2$, \\
\item If $d \geq \sqrt{2}R$, then $P_1(D \leq dR) < P_2(D \leq dR)$ $\Leftrightarrow N_1 > N_2$.
\end{itemize}
\end{proposition}

Proposition \ref{proposition3} can be used to bound the degrees of freedom in cases where the only input data are dissimilarity scores, if a process for generating uniform and independent scores is available. Moreover, the above proposition implies that as the dimension increases the pdf of Eq. (\ref{eq:hypersphere_pdf}) becomes more concentrated around $\sqrt{2}R$ (Fig. \ref{fig:pdf_examples}).

The latter is established through the following property of $\mu$:

\begin{proposition}
\label{theorem3}
$$\lim_{N\to\infty} \mu = \sqrt{2}R$$
\end{proposition}

\begin{proof}
The Stirling approximation for factorials suggests that

\begin{equation} \label{eq:stirling}
\lim_{N\to\infty} N! = \sqrt{2\pi}N^{N+\frac{1}{2}}e^{-N}
\end{equation}

Through Stirling approximation, Eq. (\ref{eq:moments_equation_3}) becomes

\begin{equation} \label{eq:moments_equation_asymptotic}
\lim_{N\to\infty} \mu = \lim_{N\to\infty} \frac{\sqrt{2\pi} [(\frac{N}{2})^{\frac{N-1}{2}}]^2 \sqrt{N-\frac{1}{2}}}{\sqrt{\pi} N^{N-\frac{1}{2}}} 2^{N-1}R
\end{equation}
i.e.,

\begin{equation} \label{eq:moments_equation_asymptotic2}
\lim_{N\to\infty} \mu = \lim_{N\to\infty} \frac{\sqrt{2} R \sqrt{N-\frac{1}{2}}}{\sqrt{N}} = \sqrt{2}R
\end{equation}
\end{proof}

It follows from propositions \ref{theorem2} and \ref{theorem3} that the variance limit when the dimension approaches infinity is zero, i.e.,

\begin{proposition}
\label{proposition4}
$$\lim_{N\to\infty} \sigma^2 = 0$$
\end{proposition}
Figures \ref{fig:mean_value} and \ref{fig:variance_value} show $\mu$ and $\sigma^2$ as a function of $N$, respectively. From proposition \ref{proposition4} it follows that as the dimension increases the hypersphere chord lengths are concentrated within a small range around $\sqrt{2}R$. In order to give a quantitative measure of the chord length range, the difference between the largest and the smallest $q$-quantile is employed. Note that this distance denotes the length of the interval from which the $\frac{q-2}{q}$ most central values are selected. The results are shown on Table \ref{tab:quantiles}. Table \ref{tab:quantiles} demonstrates that when $N \gg 2$ most chord lengths lie within a small range around $\sqrt{2}R$. For example, the $93\%$ of chord lengths of a 256-sphere lie inside an interval of length $0.1358R$ around $\sqrt{2}R$, i.e., an interval that spans only $6.79\%$ of the available distance range.

\begin{figure*}[htb]
\centering
\includegraphics[width=0.8\textwidth]{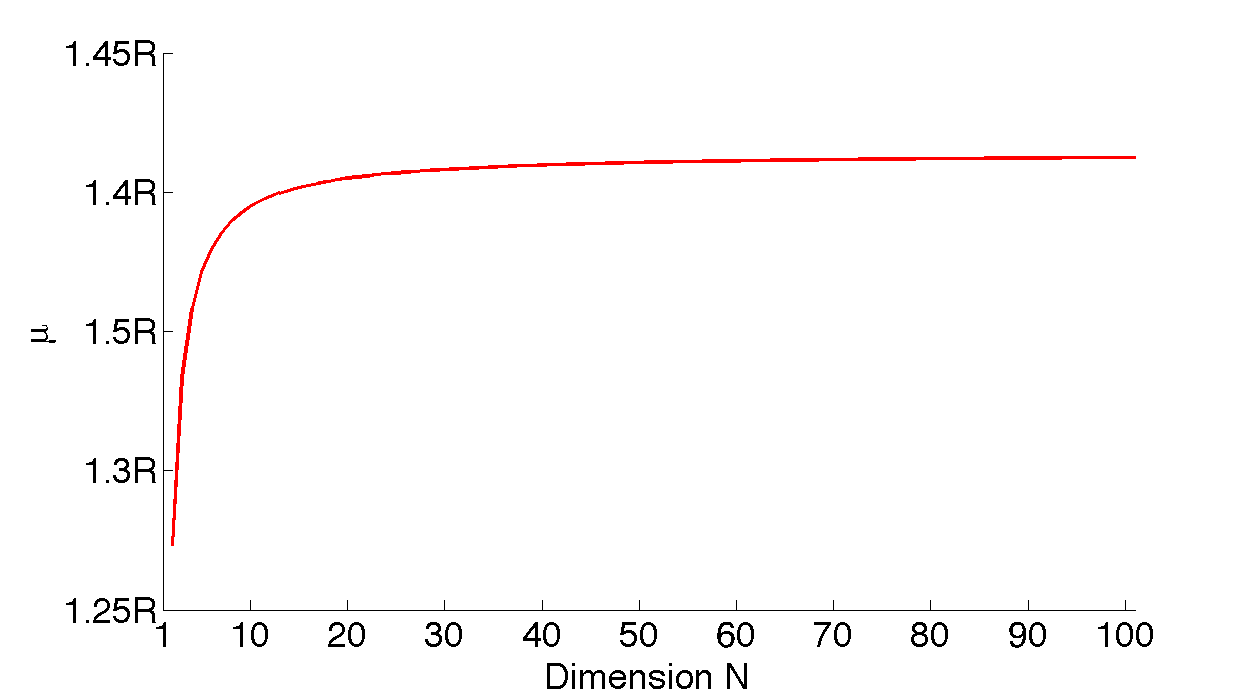}
\caption {The mean chord distance as a function of the hypersphere dimension $N$.}
\label{fig:mean_value}
\end{figure*}

\begin{figure*}[htb]
\centering
\includegraphics[width=0.8\textwidth]{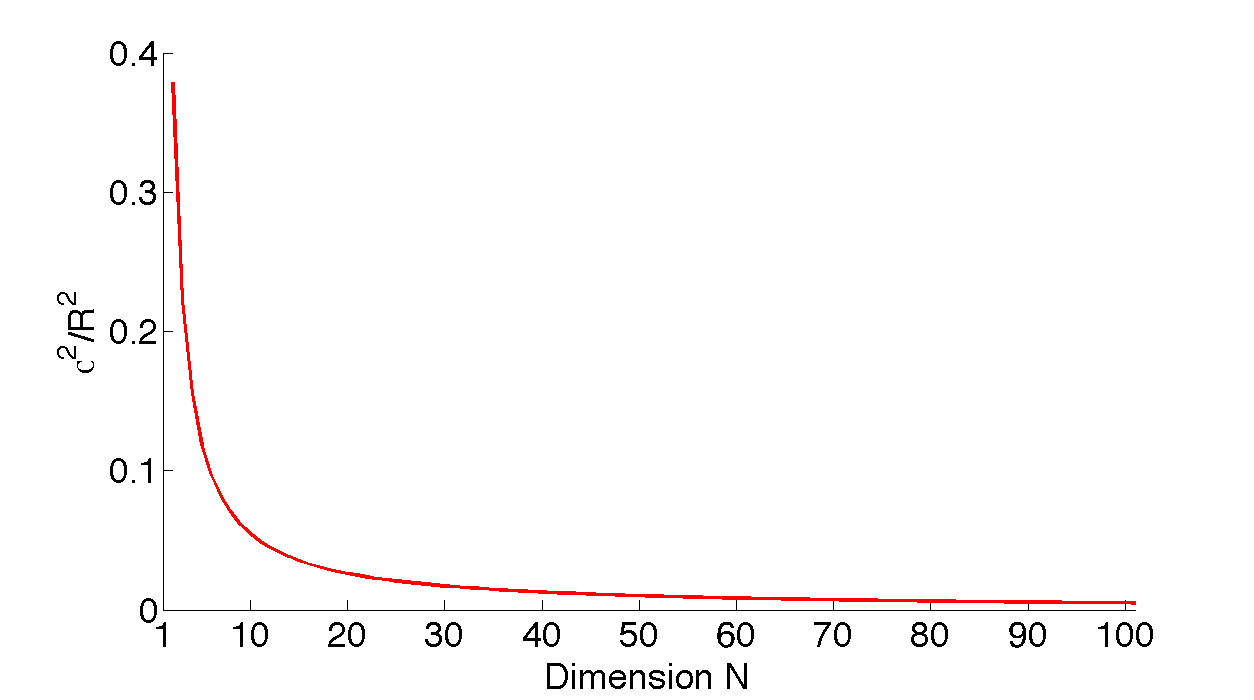}
\caption {The chord distance variance as a function of the hypersphere dimension $N$.}
\label{fig:variance_value}
\end{figure*}

\begin{table}[tb]
\label{tab:quantiles}
\begin{center}
\begin{tabular}{|c|c|c|c|c|c|}
\hline N & $q = 3$ ($33\%$) & $q = 4$ ($50\%$) & $q = 6$ ($67\%$) & $q = 8$ ($75\%$) & $q = 16$ ($93\%$) \\ \hline
2 & 0.7321 & 1.0824 & 1.4142 & 1.5714 & 1.7943 \\ \hline
3 & 0.4783 & 0.7321 & 1.0092 & 1.1637 & 1.4365 \\ \hline
4 & 0.3781 & 0.5839 & 0.8173 & 0.9534 & 1.2101 \\ \hline
8 & 0.2377 & 0.3703 & 0.5261 & 0.621 & 0.8102 \\ \hline
16 & 0.1597 & 0.2494 & 0.3563 & 0.4222 & 0.5581 \\ \hline
32 & 0.1102 & 0.1724 & 0.2468 & 0.293 & 0.389 \\ \hline
64 & 0.077 & 0.1205 & 0.1727 & 0.2052 & 0.2731 \\ \hline
128 & 0.0542 & 0.0848 & 0.1215 & 0.1444 & 0.1925 \\ \hline
256 & 0.0382 & 0.0598 & 0.0857 & 0.1019 & 0.1358 \\ \hline
\hline
\end{tabular}
\end{center}
\caption{Quantile ranges for different $N$ and $q$ (in parentheses is the percentage of distribution values that lie inside the range.} 
\end{table}

An intuitive analysis of the distribution change with $N$ can be done by considering the 3-dimensional sphere as a globe. In this case, the fixed point $p$ is on the North Pole, and the points having distance $\sqrt{2}R$ from $p$ are the points in the equator. Seen under this prism, the concentration of the distribution around $\sqrt{2}R$ denotes an expansion of the equatorial areas, in expense of the high latitude areas. Ideally, when $N \rightarrow \infty$ ``almost all'' points of the hypersphere lie in the equator (meaning that the set of points within distance $\sqrt{2}R$ have infinitely more points than the set of points with distance either larger or smaller than $\sqrt{2}R$).

\section{Unitary vector dot product distribution}
\label{sec:uvdp}

In the hypersphere of radius $1$ the distance of two points $p_1$ and $p_2$ $d_{12}$ and the dot product of the vectors $\overrightarrow{u_1} = \overrightarrow{Op_1}$ and $\overrightarrow{u_2} = \overrightarrow{Op_2}$ (where $O$ is the centre of the hypersphere) are connected via the following relation:

\begin{equation} \label{eq:distance_dotproduct}
d^2_{12} = 2-2\overrightarrow{u_1} \cdot \overrightarrow{u_2}
\end{equation}
Since the distance distribution is already known, the dot product distribution can be estimated through Eq. (\ref{eq:distance_dotproduct}). As a matter of fact, if $F_C$ is the cdf of the unitary dot product $C=\overrightarrow{u_1} \cdot \overrightarrow{u_2}$ then

\begin{equation} \label{eq:distance_dotproduct2}
F_C(c) = P(C \leq c) = P((1-\frac{d^2_{12}}{2}) \leq c) = P(d^2_{12} \geq 2-2c)
\end{equation}
$d_{12}$ values are always positive, thus the above equation can be reformulated as

\begin{equation} \label{eq:distance_dotproduct3}
F_C(c) = P(d_{12} \geq \sqrt{2-2c}) = 1 - P(d_{12} \leq \sqrt{2-2c}) = 1 - F_{d}(\sqrt{2-2c})
\end{equation}

By replacing $d = \sqrt{2-2c}$ in Eq. (\ref{eq:hypersphere_cdf}) it follows that

\begin{equation} \label{eq:unitary distance_cdf}
\begin{split}
F_C(c) = 1 - \frac{1}{2}I_{1-c^2}(\frac{N-1}{2},\frac{1}{2}), c \leq 0 \\
F_C(c) = \frac{1}{2}I_{1-c^2}(\frac{N-1}{2},\frac{1}{2}), c \geq 0
\end{split}
\end{equation}

The corresponding pdf is

\begin{equation} \label{eq:unitary distance_pdf}
\begin{split}
f_C(c) = \frac{{(1-c^2)}^{\frac{N-3}{2}}}{B(\frac{N-1}{2},\frac{1}{2})}
\end{split}
\end{equation}

The pdf is symmetric around $c=0$, thus all odd-order moments (among which, the average) are independent of $N$ and equal to 0. For the even-order moments, i.e., when $k=2\lambda$, the following formula stands:

\begin{equation} \label{eq:unitary distance_pdf}
\begin{split}
\mu_{2\lambda} = \frac{B(\lambda + \frac{1}{2}, \frac{N-1}{2})} {B(\frac{1}{2}, \frac{N-1}{2})}
\end{split}
\end{equation}

\section{Conclusions}
\label{sec:cr}

In this paper, the probability distribution that follow the $N$-sphere chord lengths was introduced. After estimating the pdf and the cdf, some of its basic properties were examined, among which its dependency from the dimension of the hypersphere. Starting from it, the distribution of the dot product of two randomly selected unitary vectors was also estimated. 


\bibliography{referernces}
\bibliographystyle{plain}

\end{document}